\documentclass[12pt,twoside]{amsart}
\usepackage{amssymb,amsmath,amsthm, amscd, enumerate, mathrsfs, bm}
\usepackage{graphicx, hhline}
\usepackage[all]{xy}

\title{On a vanishing theorem for surfaces}
\author{Osamu Fujino and Nao Moriyama}
\date{2025/11/27, version 0.01}
\subjclass[2020]{Primary 32C15; Secondary 14E30}
\keywords{vanishing theorems, algebraic surfaces, 
complex analytic surfaces, 
minimal model program}
\address{Department of 
Mathematics, Graduate School of Science, 
Kyoto University, Kyoto 606-8502, Japan}
\email{fujino@math.kyoto-u.ac.jp}
\address{Department of Mathematics, Graduate 
School of Science, Kyoto University, Kyoto 606-8502, Japan}
\email{moriyama.nao.22s@st.kyoto-u.ac.jp}
\DeclareMathOperator{\Supp}{Supp}

\newtheorem{thm}{Theorem}[section]
\newtheorem{lem}[thm]{Lemma}

\theoremstyle{definition}
\newtheorem{defn}[thm]{Definition}

\newtheorem*{ack}{Acknowledgments}

\makeatletter

\@addtoreset{equation}{section}
\makeatother

\begin{document}

\begin{abstract}
We propose a new formulation of a vanishing 
theorem for surfaces. Although this vanishing theorem 
follows easily from the well-known Kawamata--Viehweg 
vanishing theorem, it turns out to be remarkably useful. 
In particular, it is sufficient for the minimal model 
theory of log surfaces, and it allows one to carry 
out both the minimal model program and the abundance 
theorem for log surfaces without invoking any of the deeper vanishing theorems.
\end{abstract}

\maketitle 


\section{Introduction}\label{a-sec1}

This short paper aims to make the minimal model theory for log surfaces, 
as established in \cite{fujino2} and \cite{moriyama}, more accessible. 
In \cite{fujino2}, we develop the minimal model theory for algebraic 
log surfaces in characteristic zero, building on the framework 
from \cite{fujino1}. In contrast, \cite{moriyama} addresses complex analytic 
log surfaces within the framework of \cite{fujino5}. Both \cite{fujino2} 
and \cite{moriyama} rely on intricate vanishing theorems derived 
from \cite{fujino1} and \cite{fujino6}.
In this paper, we demonstrate that the vanishing theorem presented 
in Theorem \ref{a-thm1.1} follows easily from 
the well-known Kawamata--Viehweg vanishing theorem. 
This result shows that the Kawamata--Viehweg vanishing theorem 
is sufficient for the minimal model theory of log surfaces 
discussed in \cite{fujino2} and \cite{moriyama}. Consequently, 
readers primarily interested in the surface theory will be 
able to avoid the more complicated vanishing theorems 
of \cite{fujino1}, \cite{fujino5}, and \cite{fujino6}.

\begin{thm}[Main Theorem]\label{a-thm1.1}
Let $X$ be a smooth complex analytic surface and let $\Delta$ be a boundary 
$\mathbb{R}$-divisor on $X$ such that $\Supp \Delta$ is a simple normal crossing divisor. 
Let $f\colon X \to Y$ be a bimeromorphic morphism to a 
complex analytic space $Y$, and let $\pi\colon Y \to Z$ be a proper 
surjective morphism of complex analytic spaces such that 
$\psi:=\pi\circ f$ is projective. 
Let $\mathcal{L}$ be a line bundle on $X$ such that 
$\mathcal{L}-(\omega_X+\Delta)$ is $\psi$-nef 
and $\psi$-big over $Z$. 
Assume that $\bigl(\mathcal{L}-(\omega_X+\Delta)\bigr)\cdot C>0$ for every 
irreducible component $C$ of $\lfloor \Delta\rfloor$ such that 
$f(C)$ is a curve and $\psi(C)$ is a point. 
Then 
\[
R^p\pi_* R^q f_* \mathcal{L}=0
\] 
for every $p>0$ and every $q$. 
\end{thm}

Theorem \ref{a-thm1.1} 
is not included in the vanishing theorems stated 
in \cite{fujino1}, \cite[Chapter 5]{fujino3}, \cite{fujino5}, or \cite{fujino6}, 
and, to the best of our knowledge, it appears to be a new formulation. 
By Theorem \ref{a-thm1.1}, we can formulate and 
prove the Kodaira vanishing theorem, the vanishing 
theorem of Reid--Fukuda type, and 
related results for log canonical surfaces 
in both the algebraic and complex analytic 
settings, without resorting to the intricate 
vanishing theorems in \cite[Chapter~5]{fujino3}, 
\cite{fujino6}, or \cite{fujino-fujisawa}. We leave 
the details to the interested 
reader; see \cite[Theorems~5.6.4 and 5.7.6]{fujino3} for further discussion.
 
\medskip

The vanishing theorem in \cite{fujino1} is proved using the theory of 
mixed Hodge structures on cohomology with compact support.  
The work in \cite{fujino5} relies on Saito's theory of mixed Hodge modules,  
while \cite{fujino-fujisawa} provides an alternative approach to the results 
of \cite{fujino5}.  
As indicated by its title, \cite{fujino-fujisawa} employs the theory of 
variations of mixed Hodge structures.  

In contrast, in this paper we only make use of the (relative) 
Kawamata--Viehweg vanishing theorem.  
Although its validity for projective morphisms between complex analytic spaces 
is not entirely obvious, it is widely accepted as a standard result by experts.  
For readers interested exclusively in algebraic varieties,  
everything can be assumed algebraic, and the classical Kawamata--Viehweg vanishing 
can be applied throughout this paper.  

We note that the formulation of Theorem~\ref{a-thm1.1} is new,  
even in the purely algebraic setting.

\medskip

Before proceeding, let us briefly outline the structure of this short paper. 
Section~\ref{a-sec2} collects the preliminaries: we recall several 
basic definitions
and state the Kawamata--Viehweg vanishing theorem for projective morphisms 
between complex analytic spaces without proof. 
In Section~\ref{a-sec3}, we prove two results that follow rather 
easily from the Kawamata--Viehweg vanishing theorem. 
Section~\ref{a-sec4} is devoted to the proof of Theorem~\ref{a-thm1.1}. 
Finally, in Section~\ref{a-sec5}, we explain, for each relevant paper, that 
the minimal model theory for log surfaces can be carried out using only 
the Kawamata--Viehweg vanishing theorem in light of the results established here.

\section{Preliminaries}\label{a-sec2}

In this paper, we freely use the standard notation and definitions 
of the minimal model theory for algebraic varieties from \cite{fujino1}, 
and for projective morphisms of complex analytic spaces from \cite{fujino5}.

\begin{defn}[Log surfaces]\label{a-def2.1}
Let \( X \) be a normal complex analytic surface, and let \( \Delta \) be a boundary 
\( \mathbb{R} \)-divisor on \( X \) such that \( K_X + \Delta \) is \( \mathbb{R} \)-Cartier. 
Then the pair \( (X, \Delta) \) is called a \emph{log surface}. We recall that an 
\( \mathbb{R} \)-divisor is said to be a {\em{boundary \( \mathbb{R} \)-divisor}} if all its 
coefficients lie in the interval \( [0,1] \).
\end{defn}
 
We discuss in \cite{fujino2} and \cite{moriyama} that 
the minimal model program and the abundance conjecture 
hold for $\mathbb Q$-factorial 
log surfaces \((X, \Delta)\) without assuming that they are necessarily log canonical.

\begin{defn}[Nefness]\label{a-def2.2}
Let \( f \colon X \to Y \) be a projective morphism 
between complex analytic spaces, 
and let \( \mathcal{L} \) be an \( \mathbb{R} \)-line bundle, 
or the sum of an \( \mathbb{R} \)-line bundle 
and an \( \mathbb{R} \)-Cartier divisor, on \( X \). We 
say that \( \mathcal{L} \) is \emph{\( f \)-nef over \( Y \)} if 
\[
\mathcal{L} \cdot C \ge 0
\]
for every curve \( C \subset X \) such that \( f(C) \) is a point.
\end{defn}

There are various characterizations of relative bigness. 
The following one is taken from \cite[Chapter~II, 5.17.~Corollary]{nakayama2}.

\begin{defn}[Bigness]\label{a-def2.3}
Let \( f \colon X \to Y \) be a projective surjective 
morphism from a normal 
complex analytic space \( X \). Let \( \mathcal{L} \) be 
an \( \mathbb{R} \)-line bundle, 
or the sum of an \( \mathbb{R} \)-line bundle and an \( \mathbb{R} \)-Cartier 
divisor, on \( X \). 
We say that \( \mathcal{L} \) is \emph{\( f \)-big over \( Y \)} if, 
for any irreducible 
component \( F \) of an analytically sufficiently general fiber of \( f \), 
the restriction 
\( \mathcal{L}|_F \) is big in the usual sense.
\end{defn}

In this paper, we will freely use the following 
relative Kawamata--Viehweg vanishing theorem for 
projective morphisms between complex analytic spaces. 

\begin{thm}[Relative 
Kawamata--Viehweg vanishing theorem]\label{a-thm2.4}
Let $f\colon X\to Y$ be a 
projective surjective morphism from a smooth 
complex analytic space $X$ and let $\Delta$ be an $\mathbb R$-divisor 
on $X$ such that $\lfloor \Delta\rfloor =0$ and $\Supp \Delta$ is a 
simple normal crossing divisor. 
Let $\mathcal L$ be a line bundle on $X$ such that 
$\mathcal L-(\omega_X+\Delta)$ is $f$-nef and $f$-big over $Y$. 
Then $R^if_*\mathcal L=0$ holds for every $i>0$. 
\end{thm}

For the details of Theorem \ref{a-thm2.4}, see, 
for example, \cite[Theorem 3.7]{nakayama1} and 
\cite[Chapter II. 5.12.~Corollary]{nakayama2}. 

\section{Lemmas}\label{a-sec3}

We begin with a vanishing theorem of Reid--Fukuda type for surfaces, 
which is a slight generalization of the Kawamata--Viehweg vanishing theorem.  

\begin{lem}[Vanishing theorem of Reid--Fukuda type]\label{a-lem3.1}
Let $V$ be a smooth complex analytic surface, and let $\Delta$ be a boundary 
$\mathbb{R}$-divisor on $V$ such that $\Supp \Delta$ is a simple 
normal crossing divisor. 
Let $\varphi\colon V\to W$ be a projective surjective morphism 
of complex analytic spaces, and let $\mathcal{L}$ be a line bundle on $V$ such that 
$\mathcal{L}-(\omega_V+\Delta)$ is $\varphi$-nef and $\varphi$-big over $W$. 
Assume further that 
\[
\bigl(\mathcal{L}-(\omega_V+\Delta)\bigr)\cdot C>0
\]
for every irreducible component $C$ of $\lfloor \Delta\rfloor$ that is mapped to a point by $\varphi$.  
Then $R^i\varphi_*\mathcal{L}=0$ for every $i>0$. 
\end{lem}

Lemma \ref{a-lem3.1} follows easily from the relative Kawamata--Viehweg 
vanishing theorem (Theorem \ref{a-thm2.4}), but we provide a proof here for completeness. 

\begin{proof}[Proof of Lemma \ref{a-lem3.1}]
If $\lfloor \Delta \rfloor = 0$, then $R^i\varphi_*\mathcal L = 0$ for every $i>0$ by the 
relative Kawamata--Viehweg vanishing theorem.  
Thus we may assume that $\lfloor \Delta\rfloor \ne 0$.  
Without loss of generality, we may further assume that $\lfloor \Delta\rfloor$ has only finitely many irreducible components, since the statement is local over $W$. 

Take an irreducible component $C$ of $\lfloor \Delta\rfloor$. 
Consider the short exact sequence
\[
0\to \mathcal L(-C)\to \mathcal L\to \mathcal L|_C\to 0,  
\]
where $\mathcal L(-C):=\mathcal L\otimes \mathcal O_V(-C)$. 
Note that
\[
\mathcal L(-C)-(\omega_V+\Delta-C)=\mathcal L-(\omega_V+\Delta). 
\]
By induction on the number of irreducible components of $\lfloor \Delta\rfloor$, 
we obtain $R^i\varphi_*\mathcal L(-C)=0$ for every $i>0$. 

Next, observe that
\[
\mathcal L|_C-\bigl(\omega_C+(\Delta-C)|_C\bigr)
=\bigl(\mathcal L-(\omega_V+\Delta)\bigr)|_C. 
\]
Since $C$ is a curve, it follows immediately that 
$R^i\varphi_*(\mathcal L|_C)=0$ for every $i>0$. 

Therefore, the long exact sequence associated to 
\[
0\to \mathcal L(-C)\to \mathcal L\to \mathcal L|_C\to 0
\] 
yields $R^i\varphi_*\mathcal L=0$ for all $i>0$.  
This completes the proof of Lemma \ref{a-lem3.1}. 
\end{proof}

The following lemma is a straightforward application of Lemma \ref{a-lem3.1}.  
Alternatively, one may use a more elementary argument, as in \cite[Theorem 11.1]{fujino4}, 
to prove Lemma \ref{a-lem3.2}. In any case, Lemma \ref{a-lem3.2} is easy to prove, 
but it is nevertheless quite useful.

\begin{lem}\label{a-lem3.2}
Let $f\colon X\to Y$ be a projective 
bimeromorphic morphism 
from a smooth complex analytic surface $X$, 
and let $\Delta$ be a boundary $\mathbb{R}$-divisor on $X$ 
such that $\Supp \Delta$ is a simple normal crossing divisor. 
Assume that $\mathcal{L}-(\omega_X+\Delta)$ is $f$-nef over $Y$. 

If $R^1 f_* \mathcal{L}\neq 0$, then the support of 
$R^1 f_* \mathcal{L}$ is contained in $f(B)$, 
where $B$ is the union of the irreducible components 
of $\lfloor \Delta\rfloor$ that are mapped to a point by $f$. 
In particular, $R^1 f_* \mathcal{L}$ is a skyscraper sheaf on $Y$. 

Moreover, $f_* \mathcal{L}$ is a torsion-free coherent sheaf on $Y$, and 
$R^i f_* \mathcal{L}=0$ for every $i\ge 2$. 
\end{lem}

\begin{proof}[Proof of Lemma \ref{a-lem3.2}]
Consider the short exact sequence 
\[
0\to \mathcal L(-B)\to \mathcal L\to \mathcal L|_B\to 0,  
\] 
where $\mathcal L(-B):= \mathcal L\otimes \mathcal O_X(-B)$. 
Note that 
\[
\mathcal L(-B)-(\omega_X+\Delta-B)=\mathcal L-(\omega_X+\Delta). 
\] 
Since $f$ is projective bimeromorphic, $\mathcal L-(\omega_X+\Delta)$ is $f$-nef 
and $f$-big over $Y$. 
By definition, $C\to f(C)$ is finite for any irreducible component $C$ of $\lfloor 
\Delta-B\rfloor$. 
Therefore, Lemma \ref{a-lem3.1} implies 
\[
R^i f_*\mathcal L(-B)=0 \quad \text{for every } i>0. 
\] 

The associated long exact sequence then gives
\[
R^i f_*\mathcal L \simeq R^i f_*(\mathcal L|_B) = H^i(B, \mathcal L|_B)
\]
for every $i>0$.  
Hence, the support of $R^1 f_*\mathcal{L}$ is contained in $f(B)$, as claimed. 
This completes the proof of Lemma \ref{a-lem3.2}. 
\end{proof}

\section{Proof of Theorem \ref{a-thm1.1}}\label{a-sec4}

In this section, we prove Theorem \ref{a-thm1.1}. 
Theorem \ref{a-thm1.1} follows directly from Lemmas \ref{a-lem3.1} and 
\ref{a-lem3.2}, and is thus a slight generalization of 
the relative Kawamata--Viehweg vanishing theorem (see Theorem \ref{a-thm2.4}). 

\begin{proof}[Proof of Theorem \ref{a-thm1.1}]
By Lemma \ref{a-lem3.2}, it suffices to show that 
$R^p\pi_*(f_*\mathcal L)=0$ for every $p>0$. 

Let $B$ be the union of the irreducible components of 
$\lfloor \Delta\rfloor$ that are mapped to a point by $f$, and consider 
\[
0\to \mathcal L(-B)\to \mathcal L\to \mathcal L|_B\to 0, 
\] 
where $\mathcal L(-B):=\mathcal L\otimes \mathcal O_X(-B)$. 

Since $\psi=\pi\circ f$ is projective, $f$ is projective. 
By Lemma \ref{a-lem3.1} and 
\[
\mathcal L(-B)-(\omega_X+\Delta-B)=\mathcal L-(\omega_X+\Delta), 
\] 
we have $R^i f_*(\mathcal L(-B))=0$ for every $i>0$. 

Thus, 
\[
0\to f_*\mathcal L(-B)\to f_*\mathcal L\to f_*(\mathcal L|_B)\to 0
\] 
is exact, and we obtain the long exact sequence
\[
\cdots \to R^p\pi_*(f_*\mathcal L(-B))\to 
R^p\pi_*(f_*\mathcal L)\to R^p\pi_*(f_*(\mathcal L|_B))\to \cdots. 
\] 

By the definition of $B$, $f_*(\mathcal L|_B)$ is either zero or a 
skyscraper sheaf. Therefore,
\[
R^p\pi_*(f_*(\mathcal L|_B))=0
\] 
for every $p>0$. It remains to show that 
\[
R^p\pi_*(f_*\mathcal L(-B))=0 
\] 
for every $p>0$. 

Since $R^i f_*(\mathcal L(-B))=0$ for every $i>0$, 
the Leray spectral sequence yields 
\[
R^p\pi_*(f_*\mathcal L(-B))=R^p\psi_* \mathcal L(-B) 
\] 
for every $p$. 

Applying the vanishing theorem of Reid--Fukuda type (Lemma \ref{a-lem3.1}) 
to $\psi$, we obtain $R^p\psi_*\mathcal L(-B)=0$ for every $p>0$. 
Hence, $R^p\pi_*(f_*\mathcal L(-B))=0$ for every $p>0$, as desired.  
This completes the proof of Theorem \ref{a-thm1.1}.  
\end{proof}

\section{On the minimal model theory for log surfaces}\label{a-sec5}

In this section, we explain, paper by paper, that 
the minimal model theory for log surfaces 
developed in \cite{fujino1}, \cite{fujino2}, \cite{fujino5}, 
and \cite{moriyama} does not require 
the intricate vanishing theorems traditionally used in the literature. 
Theorem~\ref{a-thm1.1} and Lemma~\ref{a-lem3.2} already provide 
all the ingredients necessary for those arguments.

\subsection{On the fundamental theorems established in \cite{fujino1}}
\label{a-subsec5.1}

The minimal model theory discussed in \cite{fujino1}  
depends on the torsion-free and vanishing theorem  
(see \cite[Theorem~6.3]{fujino1}).  
Let \( f \colon Y \to X \) be a projective birational morphism  
from a smooth algebraic surface \( Y \).  
In \cite[Theorem~6.3 (i)]{fujino1}, it is shown that  
the torsion-free property, which is now called 
the strict support condition, holds for \( f \).  
This is contained in Lemma~\ref{a-lem3.2} in the current paper.  
Moreover, \cite[Theorem~6.3 (ii)]{fujino1} easily follows from  
Theorem~\ref{a-thm1.1}.  
In \cite[Section~8]{fujino1}, the torsion-free and vanishing theorems  
are applied only in the case where  
\( f \colon Y \to X \) is birational.  
Hence, the results in \cite[Section~8]{fujino1} for surfaces  
can be established solely by using  
Theorem~\ref{a-thm1.1} and Lemma~\ref{a-lem3.2}.  
Thus, we recover \cite[Theorem~9.1]{fujino1} for surfaces.

\medskip

We note the following well-known elementary property.

\begin{lem}\label{a-lem5.1}
Let \( (X,\Delta) \) be a log canonical surface and  
let \( C \) be a minimal log canonical center of \( (X,\Delta) \).  
If \( \dim C = 1 \), then \( (X,\Delta) \) is purely log terminal  
in a neighborhood of \( C \).  
In particular, \( C \) is normal and \( (C,\Delta_C) \) is kawamata log terminal,  
where \( K_C + \Delta_C := (K_X + \Delta)|_C \) by adjunction.
\end{lem}

We provide a proof of Lemma~\ref{a-lem5.1} for completeness.

\begin{proof}[Proof of Lemma~\ref{a-lem5.1}]
Assume that $\dim C = 1$.  
Since $C$ is a minimal log canonical center by assumption,  
every divisor $E$ over $X$ whose center on $X$ intersects $C$  
and is different from $C$ itself has discrepancy $a(E,X,\Delta) > -1$  
by \cite[Theorem~9.1\,(2)]{fujino1}.  
In particular, there is no divisor with discrepancy $-1$ whose center meets $C$  
other than the divisor $C$.  
Hence $(X,\Delta)$ is purely log terminal in a neighborhood of $C$.  

The remaining assertions follow immediately:  
since $(X,\Delta)$ is purely log terminal near $C$,  
by adjunction, the pair $(C,\Delta_C)$ is kawamata log terminal.  
This completes the proof.
\end{proof}

By Lemma~\ref{a-lem5.1}, when \( X \) is a surface,  
if \( W \) is a minimal log canonical center,  
then \( (W, B_W) \) is a kawamata log terminal curve,  
where \( K_W + B_W := (K_X + B)|_W \) by adjunction,  
or else \( W \) is a point.  
Thus, \cite[Theorem~11.1]{fujino1} is clear when \( X \) is a surface.  
In the proof of the non-vanishing theorem in \cite[Section~12]{fujino1},  
minimal log canonical centers that are disjoint  
from the non-lc locus are treated.  
By Lemma~\ref{a-lem5.1}, such a center is either a point or  
a kawamata log terminal curve when the ambient variety is a surface.  
Therefore, no complicated arguments are required  
when working with surfaces.  
Thus, we can verify that  
Theorem~\ref{a-thm1.1} and Lemma~\ref{a-lem3.2}  
are sufficient to establish the results of \cite{fujino1}  
for surfaces. 
We strongly recommend that the reader 
consult \cite[Section~2]{fujino1} for the 
appropriate use of the vanishing theorem for surfaces.

\subsection{On the minimal model theory for log surfaces in \cite{fujino2}}
\label{a-subsec5.2}

The Kodaira-type vanishing theorem stated in \cite[2.8]{fujino2}  
follows immediately from Theorem~\ref{a-thm1.1}.  
Note that the vanishing theorem in \cite[2.8]{fujino2}  
is used repeatedly in \cite[Section~4]{fujino2}.  
In \cite[Section~9]{fujino2}, which is an appendix,  
quasi-log structures on reducible curves are treated.  
We do not check here whether Theorem~\ref{a-thm1.1}  
and Lemma~\ref{a-lem3.2} are sufficient for the proof of  
\cite[Theorem~9.1]{fujino2}.  
We leave this as an exercise for the interested reader.  
Note that \cite[Theorem~9.1]{fujino2} is not necessary  
for the minimal model program and the abundance theorem  
for log surfaces (see \cite[Theorem~1.1]{fujino2}).

\subsection{On the cone and contraction theorem established in 
\cite{fujino5}}\label{a-subsec5.3}

In \cite{fujino5}, the cone and contraction theorem is established for 
projective morphisms between complex analytic spaces.  
Note that \cite{fujino5} is a complex analytic generalization of 
\cite{fujino1}.  
The vanishing theorem (see \cite[Theorem~3.1.5]{fujino5}),  
which is one of the main results of \cite{fujino6},  
can be recovered quickly from Theorem~\ref{a-thm1.1}  
and Lemma~\ref{a-lem3.2} when  
\( f \colon X \to Y \) is a bimeromorphic morphism  
from a smooth complex analytic surface \( X \).  
Therefore, \cite[Theorem~3.3.1]{fujino5} can be proved  
without difficulties when \( X \) is a surface.

For the proof of the non-vanishing theorem  
(see \cite[Theorem~4.1.1]{fujino5}),  
no complicated vanishing theorems are needed  
if we treat only surfaces.  
In particular, we do not need \cite[Lemma~4.1.3]{fujino5}.  
This is because any minimal log canonical center of a log canonical surface  
is either a point or a kawamata log terminal curve  
by Lemma~\ref{a-lem5.1}.  
Hence, in order to prove the results of \cite{fujino5}  
for surfaces,  
Theorem~\ref{a-thm1.1} and Lemma~\ref{a-lem3.2}  
are sufficient.

\subsection{On the minimal model theory for analytic log surfaces in \cite{moriyama}}
\label{a-subsec5.4}

In \cite{moriyama}, we utilize the vanishing 
theorem stated in \cite[Lemma 2.14]{moriyama}. 
It is not difficult to verify that the case 
where \( f \colon \widetilde{X} \to X \) is a 
bimeromorphic morphism from a smooth complex 
analytic surface \( \widetilde{X} \) is 
sufficient. In this case, \cite[Lemma 2.14]{moriyama} 
can be recovered from Theorem \ref{a-thm1.1}. Thus, 
we can avoid the more intricate vanishing results 
proved in \cite{fujino6} and \cite{fujino-fujisawa}. 

\begin{ack}
The first author was partially supported by JSPS KAKENHI Grant Numbers 
JP21H04994 and JP23K20787. 
The authors are grateful to Yohei Hada for carefully reading the manuscript 
and kindly pointing out an error in an earlier version of this paper.
\end{ack}


\end{document}